\documentclass[a4paper,reqno, 11pt]{amsart}  
\usepackage[DIV=12, oneside]{typearea}
\usepackage[utf8]{inputenc}
\usepackage[T1]{fontenc}
\usepackage[english]{babel}
\usepackage[centertags]{amsmath}
\usepackage{amstext,amssymb,amsopn,amsthm}
\usepackage{nicefrac, esint}
\usepackage{mathrsfs}
\usepackage{dsfont}
\usepackage{bbm}
\usepackage{thmtools}
\usepackage{mathtools}

\usepackage[colorlinks=true, linkcolor=blue, citecolor=black]{hyperref}
\usepackage{enumitem}

\addto\extrasenglish{}
\addto\extrasenglish{}

\declaretheorem[name=Theorem, numberwithin=section]{theorem}

\newtheorem{lemma}[theorem]{Lemma}
\newtheorem{proposition}[theorem]{Proposition}

\declaretheoremstyle[bodyfont=\normalfont]{definition-style}

\declaretheoremstyle[bodyfont=\normalfont]{remark-style}
\declaretheorem[name=Remark, style=remark-style,sibling=theorem]{remark}

\numberwithin{equation}{section}

\theoremstyle{plain}

\newcommand{\N}{\mathds{N}}
\newcommand{\R}{\mathds{R}}
\newcommand{\C}{\mathds{C}}
\newcommand{\Z}{\mathds{Z}}

\newcommand{\cA}{\mathcal{A}}

\DeclareMathOperator{\pat}{path}


\title{Fractional and non-fractional Hardy inequality on~a~lattice~$\Z^d$}

\author[Dyda]{Bart{\l}omiej Dyda}
\address{Faculty of Pure and Applied Mathematics\\ Wroc{\l}aw University 
        of Science and Technology\\
        Wybrze\.ze Wyspia\'nskiego 27,
        50-370 Wroc{\l}aw, Poland
}
\email{bdyda@pwr.edu.pl\quad dyda@math.uni-bielefeld.de}


\begin{document}


\begin{abstract}
We present simple proofs of a~discrete fractional and non-fractional Hardy inequality,
Our constants are explicit, but not optimal.
In the class of power weights, we get a~complete picture of when the non-fractional Hardy inequality
holds, for any dimension $d$ of the lattice $\Z^d$ and exponent $0<p<\infty$.
\end{abstract}

\maketitle


\section{Introduction}

Around 1920, Hardy proved \cite{zbMATH02606186}, \cite[Theorem~326]{MR944909} an inequality which may be phrased as follows,
\[
\sum_{n=1}^\infty \frac{|u(n)|^p}{n^p} < \left(\frac{p}{p-1}\right)^p \sum_{n=0}^\infty |u(n+1) - u(n)|^p \quad (\textrm{for all sequences $(u(n))$ with $u(0)=0$}),
\]
unless all $u(n)=0$. This inequality has been generalised in many directions; in this note, we study the case where the summation
over $\Z_+=\{0,1,\ldots\}$ is replaced by a~summation over $\Z_+^d$ or~$\Z^d$.
The right side of such an inequality is (for $p=2$) closely connected with the quadratic form
of the graph Laplacian on the lattice $\Z^d$.

Let us first state one of our two main theorems.
In this theorem, the symbol $k\sim j$ means that $k$ and $j$ are \emph{neighbours}, see Section~\ref{sec:proofs} for a~precise definition.

\begin{theorem}\label{thm:classH}
  \begin{enumerate}
  \item
    If $0<p\leq 1<d$, then there exists a~constant $c=c(p)<\infty$ such that
    \begin{equation}\label{eq:thmlapl-p-small}
    \sum_{j\in \Z_+^d \setminus\{0\}} \frac{|u(j)|^p}{\|j\|_\infty}
    \leq c d^{-1}
    \sum_{j\in \Z_+^d \setminus\{0\}} \sum_{\substack{k\in \Z_+^d\setminus\{0\} \\ k\sim j}} |u(j)-u(k)|^p
    \end{equation}
    for all functions $u:\Z_+^d \to \C$ for which the left side of \eqref{eq:thmlapl-p-small} is finite.
  \item
    If $1\leq p <d$ and $d-p\geq \delta>0$, then there exists a~constant $c=c(p,\delta)<\infty$ such that
    \begin{equation}\label{eq:thmlapl-p-medium}
    \sum_{j\in \Z_+^d \setminus\{0\}} \frac{|u(j)|^p}{\|j\|_\infty^{p}}
    \leq c d^{p-2}
    \sum_{j\in \Z_+^d \setminus\{0\}} \sum_{\substack{k\in \Z_+^d\setminus\{0\} \\ k\sim j}} |u(j)-u(k)|^p
    \end{equation}
    for all functions $u:\Z_+^d \to \C$ for which the left side of \eqref{eq:thmlapl-p-medium} is finite.
  \item\label{item:dleqp}
    If $d<p$, then there exists a~constant $c=c(d,p)<\infty$ such that
    \begin{equation}\label{eq:thmlapl-p-large}
    \sum_{j\in \Z_+^d \setminus\{0\}} \frac{|u(j)|^p}{\|j\|_\infty^{p}}
    \leq c
    \sum_{j\in \Z_+^d\setminus\{0\} } \sum_{\substack{k\in \Z_+^d \\ k\sim j}} |u(j)-u(k)|^p
    \end{equation}
    for all functions $u:\Z_+^d \to \C$ such that $u(0)=0$.
  \item\label{item:dp}
    If $d=p$ and $\varepsilon>0$, then there exists a~constant $c=c(d,p,\varepsilon)<\infty$ such that
    \begin{equation}\label{eq:thmlapl-p-d}
    \sum_{j\in \Z_+^d \setminus\{0\}} \frac{|u(j)|^p}{\|j\|_\infty^{p + \varepsilon}}
    \leq c
    \sum_{j\in \Z_+^d \setminus\{0\}}  \sum_{\substack{k\in \Z_+^d \\ k\sim j}} \frac{|u(j)-u(k)|^p}{\|j\|_\infty^{\varepsilon}}
    \leq c
    \sum_{j\in \Z_+^d\setminus\{0\} } \sum_{\substack{k\in \Z_+^d \\ k\sim j}} |u(j)-u(k)|^p
    \end{equation}
    for all functions $u:\Z_+^d \to \C$ such that $u(0)=0$.
  \item\label{item:deq1_ple1}
    If $d=1$, $0<p<1$ and $\varepsilon>0$, then there exists a~constant $c=c(p,\varepsilon)<\infty$ such that
    \begin{equation}\label{eq:thmlapl-p-le-d-eq-1}
    \sum_{j\in \Z_+ \setminus\{0\}} \frac{|u(j)|^p}{\|j\|_\infty^{1 + \varepsilon}}
    \leq c
    \sum_{j\in \Z_+ } \sum_{\substack{k\in \Z_+ \\ k\sim j}} |u(j)-u(k)|^p
    \end{equation}
    for all functions $u:\Z_+ \to \C$ such that $u(0)=0$.
  \item\label{item:optimal}
    Exponents $1$, $p$, $p$ at $\|j\|_\infty$ on left sides of  \eqref{eq:thmlapl-p-small},
    \eqref{eq:thmlapl-p-medium}, \eqref{eq:thmlapl-p-large} (respectively) are optimal,
    that is, they cannot be replaced by smaller exponents.
    Furthermore, (\ref{item:dp}) and (\ref{item:deq1_ple1}) fail to hold for $\varepsilon=0$.
    Finally, the exponent $-1$ at $d$ on right side of \eqref{eq:thmlapl-p-small} is optimal,
    i.e., cannot be replaced by a~smaller one.
\item\label{item:Z}
    All six results above hold also for $\Z^d$ in place of $\Z_+^d$ (perhaps with bigger constants).
  \end{enumerate}
\end{theorem}

\begin{remark}
  We note that in \eqref{eq:thmlapl-p-small} and \eqref{eq:thmlapl-p-medium} the value  $u(0)$ does not appear, still
  we formulate the result for functions defined on $\Z_+^d$ and not on $\Z_+^d \setminus\{0\}$ for more symmetry.
\end{remark}

One dimensional case ($d=1$ in (\ref{item:dleqp})) of Theorem~\ref{thm:classH}
is well-known and goes back to Hardy, while (\ref{item:deq1_ple1}) is rather trivial and included
only for completeness.
The ultimate results here are \cite{MR3779222} (for $p=2$) and  \cite{MR4648598} (for general $p>1$),
see also \cite{zbMATH07900724},
where the optimal weight (with optimal constant) in \eqref{eq:thmlapl-p-large} is obtained.
Weighted optimal one-dimensional discrete Hardy inequalities are studied in \cite{zbMATH07545074, zbMATH07809011}. On the other hand, multidimensional discrete Hardy inequalities, for $p=2$ and
functions vanishing on all coordinate axis, are studied in \cite{zbMATH07341408}.

We note that the case of $p=2$ and $d\geq 3$ of Theorem~\ref{thm:classH} is proved
in \cite{MR3584186}, while in case $p=d=2$, a~logarithmic Hardy inequality stronger than our
case (\ref{item:dp}) is obtained.
These results are refined in \cite{MR4630769, MR4412975, zbMATH07642047} to obtain asymptotics of the constants when $d\to \infty$.

 There are also papers \cite{MR4597627} and \cite{MR4768506}, see especially Example~6.4 in the latter article,
which concerns the case $1<p<d$. There the Hardy inequality with the optimal weight is obtained,
however, the weight is expressed by the Green function and is not explicit, in particular,
it is not clear whether it can be bounded by a~power function.

Consequently, the cases (1) and (3) of Theorem~\ref{thm:classH} seem to be completely new. Case (2) is probably new for $p\neq 2$ and $d>1$, at least with this form of the weight.

The next theorem is a~fractional counterpart of Theorem~\ref{thm:classH}. Similarly as before,
we also obtain power weights for every $0<p<\infty$, $d\in\N$ and $s>0$, however here we do not prove
optimality of exponents.

\begin{theorem}\label{thm:fracH}
  Let $0 < s,p < \infty$ and $sp\neq d$.
  \begin{enumerate}
  \item
    If $sp<d$ and $d-sp\geq \delta>0$, then there exists a~constant $c=c(s,p,\delta)<\infty$ such that
    \begin{equation}\label{eq:thmfrac-sp-small}
    \sum_{j\in \Z_+^d \setminus\{0\}} \frac{|u(j)|^p}{\|j\|_\infty^{sp}}
    \leq c
    \sum_{j\in \Z_+^d \setminus\{0\}} \sum_{\substack{m\in \Z_+^d \setminus\{0\}\\ m\neq j}} \frac{|u(j)-u(m)|^p}{\|j-m\|_\infty^{sp+d}}
    \end{equation}
    for all functions $u:\Z_+^d \to \C$ for which the left side of \eqref{eq:thmfrac-sp-small} is finite.
  \item
    If $sp=d$ and $\varepsilon>0$, then there exists a~constant $c=c(d,s,p,\varepsilon)<\infty$ such that
    \begin{equation}\label{eq:thmfrac-sp-d}
    \sum_{j\in \Z_+^d \setminus\{0\}} \frac{|u(j)|^p}{\|j\|_\infty^{sp+\varepsilon}}
    \leq c
    \sum_{j\in \Z_+^d } \sum_{\substack{m\in \Z_+^d\\ m\neq j}} \frac{|u(j)-u(m)|^p}{\|j-m\|_\infty^{sp+d+\varepsilon}}
    \leq c
    \sum_{j\in \Z_+^d } \sum_{\substack{m\in \Z_+^d\\ m\neq j}} \frac{|u(j)-u(m)|^p}{\|j-m\|_\infty^{sp+d}}
    \end{equation}
    for all functions $u:\Z_+^d \to \C$ such that $u(0)=0$.
  \item
    If $sp>d$, then there exists a~constant $c=c(d,s,p)<\infty$ such that
    \begin{equation}\label{eq:thmfrac-sp-large}
    \sum_{j\in \Z_+^d \setminus\{0\}} \frac{|u(j)|^p}{\|j\|_\infty^{sp}}
    \leq c
    \sum_{j\in \Z_+^d } \sum_{\substack{m\in \Z_+^d\\ m\neq j}} \frac{|u(j)-u(m)|^p}{\|j-m\|_\infty^{sp+d}}
    \end{equation}
    for all functions $u:\Z_+^d \to \C$ such that $u(0)=0$.
  \item
    All three above results hold also for $\Z^d$ in place of $\Z_+^d$ (perhaps with bigger constants).
  \end{enumerate}
\end{theorem}

\begin{remark}
  We emphasize that in above theorem we do not assume that $s<1$.
\end{remark}

Inequality (1) of the above Theorem, in the case $\Z$ (i.e., $d=1$), $p=2$ and $0<s<\frac{1}{2}$, was proved in \cite{MR3882021}.
There it is also proved that the right side of \eqref{eq:thmfrac-sp-small} is comparable with the quadratic form
corresponding to the fractional discrete Laplace operator on~$\Z$ with exponent~$s$.
We stress the fact that in the discrete setting, the kernel of fractional Laplace operator is only comparable with $|j-m|^{-d-2s}$,
see also \cite{zbMATH06893246, zbMATH07194174, zbMATH06856747}.
The optimal constant (again, in the case $d=1$, $p=2$ and $0<s<\frac{1}{2}$) was obtained in \cite{zbMATH07710867}.
The multidimensional part of Theorem~\ref{thm:fracH} seems to be new.

An interesting feature of our proof is that we conclude  parts (1)--(5) of the first Theorem,
as well as the second Theorem, from a~common result,
Lemma~\ref{lem:main}. In a~sense, this lemma is a~primary form of the Hardy inequality,
as both local and non-local type Hardy inequalities can be obtained from it.

\section{Notation and lemmas}\label{sec:proofs}
For $x\in \C^d$ we put,
\begin{align*}
  \|x\|_p &= \left(\sum_{k=1}^d |x_k|^p\right)^{1/p}, \quad 0 < p < \infty;\\
  \|x\|_\infty &= \max_{k=1,\ldots,d} |x_k|;\\
  |x| &= \|x\|_2.
\end{align*}
We say that $x,y\in \Z^d$ are \emph{neighbours}, denoted by $x\sim y$, if $|x-y|=1$.

Let
\[
A_n = \begin{cases}
  \{0\}, & \textrm{for $n=0$,}\\
  \Z_+^d \cap [0, 2^n-1]^d \setminus [0, 2^{n-1}-1]^d, & \textrm{for $n=1,2,\ldots$.}
\end{cases}
\]
For $a,b\in \R$ we denote $a\wedge b=\min(a,b)$ and $a\vee b=\max(a,b)$.

The next lemma is an important technical result, which allows us to deduce most parts of both main
theorems.
\begin{lemma}\label{lem:main}
  Let $0<s,p<\infty$ with $sp\neq d$. Let $K\in \N$ be such that
  \begin{equation}\label{eq:lemgamma}
  2^{sp+1}(2^{p-1}\vee 1) 2^{-K|sp-d|} \leq 1,
  \end{equation}
  and let
\begin{equation}\label{eq:const-lemma}
C(d,p,s,K)=\frac{2^{sp+1+K(d\wedge sp)} (2^{p-1}\vee 1)}{1-2^{-d}}.
\end{equation}
  
  If $sp<d$, then
\begin{equation}\label{eq:lemH-small-sp}
  \sum_{j\in \Z_+^d\setminus\{0\}} \frac{|u(j)|^p}{\|j\|_\infty^{sp}} \leq C(d,p,s,K)
  \sum_{n=1}^\infty \sum_{j\in A_n} \sum_{m\in A_{n+K}} \frac{|u(j)-u(m)|^p}{2^{(n+K)(d+sp)}}
\end{equation}
for all functions $u:\Z_+^d\to\C$ such that the left side of \eqref{eq:lemH-small-sp} is finite.

If $d<sp$, then
    \begin{equation}\label{eq:lemH-large-sp}
      \sum_{\substack{j\in \Z_+^d\\ 2^{K} \leq \|j\|_\infty}} \frac{|u(j)|^p}{\|j\|_\infty^{sp}}
    \leq
    C(d,p,s,K) \sum_{n=1}^\infty \sum_{j\in A_n} \sum_{m\in A_{n+K}}
    \frac{|u(j)-u(m)|^p}{2^{(n+K)(d+sp)}}
    + \sum_{\substack{j\in \Z_+^d\\ 1 \leq \|j\|_\infty < 2^{K}}} \frac{|u(j)|^p}{\|j\|_\infty^{sp}}
    \end{equation}
    for all functions $u:\Z_+^d\to\C$.
\end{lemma}

\begin{remark}
We note that the value of $u$ at $0\in \Z_+^d$ appears neither in \eqref{eq:lemH-small-sp} nor in \eqref{eq:lemH-large-sp}.
\end{remark}

\begin{proof}
  We observe that for $n=1,2,\ldots$,
\[
\#A_n = 2^{nd}(1-2^{-d}); \qquad    2^{n-1} \leq \|j\|_\infty \leq 2^n-1, \quad \textrm{for $j\in A_n$.}
\]
 We choose $k=K$ if $sp<d$ and $k=-K$ otherwise.
Let  $n\geq 1$ be such that $n+k>0$.
    Then
    \begin{align}
      \sum_{j\in A_n} \frac{|u(j)|^p}{\|j\|_\infty^{sp}}
      &\leq
      2^{(1-n)sp} \sum_{j\in A_n} |u(j)|^p \label{eq:sumAn}\\
      &\leq
      \frac{2^{(1-n)sp} (2^{p-1}\vee 1)}{\#A_{n+k}} \sum_{m\in A_{n+k}} \sum_{j\in A_n}
      \left( |u(j)-u(m)|^p + |u(m)|^p \right)\nonumber\\
      &\leq
      \frac{2^{(1-n)sp} (2^{p-1}\vee 1)}{2^{(n+k)d}(1-2^{-d})} \sum_{m\in A_{n+k}} \sum_{j\in A_n}
      |u(j)-u(m)|^p \nonumber\\
      &\phantom{\leq}+
      \frac{2^{(1-n)sp} (2^{p-1}\vee 1) 2^{nd}(1-2^{-d})}{2^{(n+k)d}(1-2^{-d})} \sum_{m\in A_{n+k}} 
      \frac{|u(m)|^p}{\|m\|_\infty^{sp}} 2^{(n+k)sp}\nonumber\\
      &=
      \frac{2^{(k+1)sp} (2^{p-1}\vee 1)}{1-2^{-d}} \sum_{m\in A_{n+k}} \sum_{j\in A_n}
      \frac{|u(j)-u(m)|^p}{2^{(n+k)(d+sp)}} \nonumber\\
      &\phantom{\leq}+
      2^{sp} (2^{p-1}\vee 1)
      \cdot 2^{k(sp-d)}
      \sum_{m\in A_{n+k}}
      \frac{|u(m)|^p}{\|m\|_\infty^{sp}}. \nonumber
    \end{align}
    Note that by \eqref{eq:lemgamma} and the definition of $k$ it holds    
    \[
     2^{sp} (2^{p-1}\vee 1)
    \cdot 2^{k(sp-d)} \leq \frac{1}{2}.
    \]
    We consider two cases.
    \subparagraph{\bf Case $sp<d$.}
    In this case, $k\geq 1$.
    We sum \eqref{eq:sumAn} over $n\geq 1$ and obtain
    \[
    \sum_{j\in \Z_+^d\setminus\{0\}} \frac{|u(j)|^p}{\|j\|_\infty^{sp}}
    \leq
    \frac{C(d,p,s,K)}{2} \sum_{n=1}^\infty \sum_{m\in A_{n+k}} \sum_{j\in A_n}
    \frac{|u(j)-u(m)|^p}{2^{(n+k)(d+sp)}}
    + \frac{1}{2} \sum_{j\in \Z_+^d\setminus\{0\}} \frac{|u(j)|^p}{\|j\|_\infty^{sp}},
    \]
    and inequality \eqref{eq:lemH-small-sp} follows.

    \subparagraph{\bf Case $sp>d$.}
    In this case, $k\leq -1$. We sum \eqref{eq:sumAn} over $-k < n \leq N$ and obtain,
    \begin{equation}\label{lem:proof3sums}
      \sum_{\substack{j\in \Z_+^d\\ 2^{-k} \leq \|j\|_\infty < 2^N}} \frac{|u(j)|^p}{\|j\|_\infty^{sp}}
    \leq
    \frac{c(d,p,s,k)}{2} \sum_{n=-k+1}^N \sum_{m\in A_{n+k}} \sum_{j\in A_n}
    \frac{|u(j)-u(m)|^p}{2^{(n+k)(d+sp)}}
    + \frac{1}{2} \sum_{\substack{j\in \Z_+^d\\ 1 \leq \|j\|_\infty < 2^{N+k}}} \frac{|u(j)|^p}{\|j\|_\infty^{sp}},
    \end{equation}
    where
    \[
    c(d,p,s,k)=\frac{2\cdot 2^{(k+1)sp} (2^{p-1}\vee 1)}{1-2^{-d}}.
    \]
    All sums in \eqref{lem:proof3sums} are finite, therefore
    \[
      \sum_{\substack{j\in \Z_+^d\\ 2^{-k} \leq \|j\|_\infty < 2^N}} \frac{|u(j)|^p}{\|j\|_\infty^{sp}}
    \leq
    c(d,p,s,k) \sum_{n=-k+1}^N \sum_{m\in A_{n+k}} \sum_{j\in A_n}
    \frac{|u(j)-u(m)|^p}{2^{(n+k)(d+sp)}}
    + \sum_{\substack{j\in \Z_+^d\\ 1 \leq \|j\|_\infty < 2^{-k}}} \frac{|u(j)|^p}{\|j\|_\infty^{sp}}.
    \]
    Letting $N\to \infty$ and rearraging the triple sum above gives
    \begin{align*}
      \sum_{\substack{j\in \Z_+^d\\ 2^{-k} \leq \|j\|_\infty}} &\frac{|u(j)|^p}{\|j\|_\infty^{sp}}
    \leq
    c(d,p,s,k) \sum_{n=1}^\infty \sum_{m\in A_n} \sum_{j\in A_{n+|k|}}
    \frac{|u(j)-u(m)|^p}{2^{n(d+sp)}}
    + \sum_{\substack{j\in \Z_+^d\\ 1 \leq \|j\|_\infty < 2^{-k}}} \frac{|u(j)|^p}{\|j\|_\infty^{sp}}\\
    &=
    c(d,p,s,k) 2^{|k|(d+sp)} \sum_{n=1}^\infty \sum_{j\in A_n} \sum_{m\in A_{n+|k|}}
    \frac{|u(j)-u(m)|^p}{2^{(n+|k|)(d+sp)}}
    + \sum_{\substack{j\in \Z_+^d\\ 1 \leq \|j\|_\infty < 2^{-k}}} \frac{|u(j)|^p}{\|j\|_\infty^{sp}}.
    \end{align*}
    Recalling that $k=-K$ we see that inequality \eqref{eq:lemH-large-sp} holds.
\end{proof}

Next result will allow us to deduce local  Hardy inequality from Lemma~\ref{lem:main}.

\begin{proposition}\label{prop:frac-class}
Let $k\in \N$. Then for every function $u:\Z_+^d\to \C$ it holds,
\begin{equation}\label{eq:prop-frac-class}
  \sum_{n=1}^\infty \sum_{j\in A_n} \sum_{m\in A_{n+k}} \frac{|u(j)-u(m)|^p}{2^{(n+k)(d+sp)}}
  \leq 
  C_L(k,s,p)\;   d^{(p\vee 1)-2} \!\!
  \sum_{x \in \Z_+^d\setminus\{0\}}
  \sum_{\substack{y \in \Z_+^d\setminus\{0\} \\ y\sim x}}  \frac{|u(x)-u(y)|^p}{\|x\|_\infty^{sp-(p \vee 1)}}\;,
 \end{equation}
where $C_L(k,s,p)$ is given by \eqref{eq:cksp}.
\end{proposition}

\begin{proof}
  For any $j\in \Z_+^d$, let $F(j)$ denote the index of the first largest coordinate of $j$, that is,
  $F(j)\in \{1,\ldots,d\}$ with $j_{F(j)} = \|j\|_\infty$ and $j_k < j_{F(j)}$ for $k<F(j)$.
  For $j\in A_n$ and $m\in A_{n+k}$, we denote by $\pat(j,m)$ the sequence of points,
  \begin{equation}\label{eq:path}
  \pat(j,m) = (x_0, x_1, \ldots, x_N)
  \end{equation}
  with the following properties:
  \begin{enumerate}
  \item $x_0=j$ and $x_N=m$;
  \item $|x_t - x_{t+1}| = 1$ for every $t=0,1,\ldots, N-1$;
  \item for each $q=1,\ldots,d$,
    the set of differences $\{x_0-x_1, \ldots, x_{N-1}-x_N\}$ does not contain $\{e_q, -e_q\}$;
  \item if $x_t-x_{t+1} \in \{e_q, -e_q\}$ for some $q\neq F(j)$, then
    $\{x_0-x_1, \ldots, x_t-x_{t+1}\} \subset \{ \pm e_1, \pm e_2, \ldots, \pm e_q\} \setminus \{ \pm e_{F(j)} \}$
    and
    $\{x_t-x_{t+1}, \ldots, x_{N-1}-x_N\} \subset \{ \pm e_q, \pm e_{q+1}, \ldots, \pm e_d \} \cup
    \{ \pm e_{F(j)} \}$.
  \item if $x_t-x_{t+1} \in \{e_{F(j)}, -e_{F(j)}\}$, then
    $\{x_t-x_{t+1}, \ldots, x_{N-1}-x_N\} \subset \{\pm e_{F(j)} \}$.
  \end{enumerate}
  In other words, $\pat(j,m)$ is the sequence of points connecting $j$ with $m$ in such a~way
  that we first go along the first coordinate axis other than $F(j)$, then along the second other than~$F(j)$, and so on, afterwards we go along $F(j)$-th coordinate; and the length of our route is the shortest possible.
  For example, in $\Z_+^3$,
  \[
  \pat((1,1,0), (0,0,2)) = ( (1,1,0), (1,0,0), (1,0,1), (1,0,2), (0,0,2) ).
  \]
  We denote by $|\pat(j,m)|$ the number $N$ in \eqref{eq:path}, and by $\pat(j,m)_t$ its $t$th element, starting from zero; so
  \[
  \pat(j,m) = (\pat(j,m)_0, \pat(j,m)_1, \ldots, \pat(j,m)_{N})
  \qquad \textrm{with $N=|\pat(j,m)|$.}
  \]
  It is easy to see that
  \[
  |\pat(j,m)| = \sum_{q=1}^d |j_q - m_q| \leq 2^{n+k} d
  \]
  and
  \begin{equation}\label{eq:path-monotonicity}
    \|\pat(j,m)_{t-1}\|_\infty \leq     \|\pat(j,m)_t\|_\infty\qquad \textrm{for $t=1,2,\ldots,|\pat(j,m)|$.}
  \end{equation}
  Therefore
  \begin{align}
    &\sum_{j\in A_n} \sum_{m\in A_{n+k}} |u(j)-u(m)|^p \nonumber\\
    &\leq
    ((2^{n+k} d)^{p-1} \vee 1) \sum_{j\in A_n} \sum_{m\in A_{n+k}}
    \sum_{t=0}^{|\pat(j,m)|-1} \left|u\Big(\pat(j,m)_t\Big) - u\Big(\pat(j,m)_{t+1}\Big)\right|^p.
    \label{eq:sumoverpath}
  \end{align}
  We will rearrange the above sum. To this end, let us fix some pair of points
  $x$, $y \in \cA_n^{n+k} := \bigcup_{t=n}^{n+k} A_t$  with $|x-y|=1$.
  Suppose that for some $j\in A_n$ and $m\in A_{n+k}$, the path from $j$ to $m$
  contains both $x$ and $y$. Let $x-y = \pm e_q$, that is, suppose
  that $x$ and $y$ differ on $q$th coordinate.
  Let us consider three cases.

    {\bf Case $q<F(j)$.}
  Then, by the definition of the path, the first $q-1$ coordinates of $x$ are the same as of $m$
  (because from $x$ to $y$ we move along $q$th axis and $q<F(j)$, hence we must have reached our destination $m$ at the coordinates
  smaller than $q$), and the last $d-q$ coordinates are the same as of $j$ (because we have not yet started to move along those coordinates axes); that is,
  \[
  x = (m_1, \ldots, m_{q-1}, x_q, j_{q+1},\ldots, j_d).
  \]
  Therefore, in this case there are at most $2^{qn}$ points $j\in A_n$ and at most $2^{(d-q+1)(n+k)}$ points $m\in A_{n+k}$
  such that $x\in \pat(j,m)$, and such that $x-y=\pm e_q$.

  {\bf Case $q=F(j)$.}
  Then, since the movement along the special $F(j)$-th coordinate is made at the very end, it holds,
  \[
  x = (m_1, \ldots, m_{q-1}, x_q, m_{q+1},\ldots, m_d).
  \]
  Therefore, in this case there are at most $2^{dn}$ points $j\in A_n$ and at most $2^{n+k}$ points $m\in A_{n+k}$
  such that $x\in \pat(j,m)$, and such that $x-y=\pm e_q$.

    {\bf Case $q>F(j)$.}
    Then
    \[
    x = (m_1, \ldots, j_{F(j)}, \ldots, m_{q-1}, x_q, j_{q+1},\ldots, j_d),
    \]
    that is, $q-2$ of the first $q-1$ coordinates of $x$ are the same as those in $m$,
    with the exception of $F(j)$-th coordinate, for which $x_{F(j)} = j_{F(j)}$;
    note that necessarily $q\geq 2$.
    Therefore, in this case for each fixed $F(j)\in \{1,\ldots,q-1\}$,
    there are at most $2^{(q-1)n}$ points $j\in A_n$ and at most $2^{(d-q+2)(n+k)}$ points $m\in A_{n+k}$
  such that $x\in \pat(j,m)$, and such that $x-y=\pm e_q$.

  Consequently, for a~given coordinate index~$q$,
  there are at most
  \begin{align}\label{eq:pairsofpoints}
    P &= 2^{qn}2^{(d-q+1)(n+k)} + 2^{dn}2^{n+k} + 2^{(q-1)n}2^{(d-q+2)(n+k)}\max(q-1,0) \\
    &= 2^{(d+1)n} 2^{kd} \left(2^{k(1-q)} + 2^{k(1-d)} + 2^{k(2-q)}\max(q-1,0)\right) \nonumber
  \end{align}
    pairs of points $(j,m)$ such that
    $x\in \pat(j,m)$ and the next point in~the path is $x\pm e_q$.
    Using the bounds $\max\{x, 0\} \leq 2^{x/2 + 1}$, $k,q \geq 1$ and $d\geq q$, we may estimate,
    \begin{equation}\label{eq:P-estimate}
    P \leq  2^{(d+1)n} 2^{kd} \cdot 6\cdot 2^{(1-q)/2} \leq  2^{(d+1)n} 2^{kd} \cdot 6.
    \end{equation}
    On the other hand, for any $j\in A_n$, $m\in A_{n+k}$ and $t=0,1,\ldots,|\pat(j,m)|-1$,
    it holds $\pat(j,m)_t \in \cA_n^{n+k}$, because of \eqref{eq:path-monotonicity}.
  Hence
  \begin{align*}
    &\sum_{j\in A_n} \sum_{m\in A_{n+k}} |u(m)-u(j)|^p\\
    &\leq 6((2^{n+k} d)^{p-1} \vee 1) \cdot
    2^{(d+1)n} 2^{kd}
    \sum_{x\in \cA_n^{n+k}}\; \sum_{\substack{y\in \cA_n^{n+k}\\y\sim x}}|u(x)-u(y)|^p 
  \end{align*}
By summing over $n$ and rearranging we obtain,
 \begin{align}
   \sum_{n=1}^\infty \sum_{j\in A_n} \sum_{m\in A_{n+k}} &\frac{|u(m)-u(j)|^p}{2^{(n+k)(d+sp)}}
   \leq
  6 \sum_{n=1}^\infty
   \frac{(2^{n+k} d)^{(p\vee 1)-1} \cdot 2^{(d+1)n}2^{kd}}{2^{(n+k)(d+sp)}}
    \sum_{x\in \cA_n^{n+k}}\; \sum_{\substack{y\in \cA_n^{n+k} \\ y\sim x}}|u(x)-u(y)|^p \nonumber\\
   &=
   \frac{6(2^kd)^{(p\vee 1)-1}}{2^{ksp}}
   \sum_{n=1}^\infty
    \sum_{x\in A_n}\; \sum_{\substack{y\in \Z_+^d\setminus\{0\} \\y \sim x}}|u(x)-u(y)|^p 
   \sum_{t=(n-k)\vee 1}^n  2^{t((p\vee 1)-sp)}.  \label{eq:proofAn}
 \end{align}
 We estimate the last sum. When $sp = p\vee 1$, then
 \[
\sum_{t=(n-k)\vee 1}^n  2^{t((p\vee 1)-sp)} \leq k+1.
 \]
 If $sp < p\vee 1$ and $x\in A_n$, then
 \[
 \sum_{t=(n-k)\vee 1}^n  2^{t((p\vee 1)-sp)}
 \leq
 \sum_{t=-\infty}^n  2^{t((p\vee 1)-sp)} = \frac{2^{n((p\vee 1)-sp)}}{1 - 2^{sp - (p\vee 1)}}
 \leq
  \frac{2^{(p\vee 1)-sp}}{1 - 2^{sp - (p\vee 1)}} \|x\|_\infty^{(p\vee 1)-sp}.
 \]
Finally, if $sp > p\vee 1$ and $x\in A_n$, then
 \[
 \sum_{t=(n-k)\vee 1}^n  2^{t((p\vee 1)-sp)}
 \leq
 \sum_{t=n-k}^\infty  2^{t((p\vee 1)-sp)} = \frac{2^{(n-k)((p\vee 1) -sp)}}{1 - 2^{(p\vee 1) -sp}}
 \leq
  \frac{2^{-k(p\vee 1)+ksp}}{1 - 2^{(p\vee 1) - sp}} \|x\|_\infty^{(p\vee 1)-sp}.
 \]
 Coming back to \eqref{eq:proofAn}, we arrive at
 \[
\sum_{n=1}^\infty \sum_{j\in A_n} \sum_{m\in A_{n+k}} \frac{|u(m)-u(j)|^p}{2^{(n+k)(d+sp)}}
\leq
\frac{C_L(k,s,p)}{5}  d^{(p\vee 1)-1}
\sum_{x \in \Z_+^d\setminus\{0\}}
\; \sum_{\substack{y\in \Z_+^d\setminus\{0\} \\y \sim x}}
\frac{|u(x)-u(y)|^p}{\|x\|_\infty^{sp-(p \vee 1)}}
 \]
  with
 \begin{equation}\label{eq:cksp}
 \frac{C_L(k,s,p)}{5} =  2^{k((p\vee 1)-sp-1)} \cdot 6\begin{cases}
   \frac{2^{(p\vee 1)-sp}}{1 - 2^{sp - (p\vee 1)}}  & \textrm{if $sp < p\vee 1$;}\\
   (k+1)  & \textrm{if $sp = p\vee 1$;}\\
    \frac{2^{-k(p\vee 1)+ksp}}{1 - 2^{(p\vee 1) - sp}}  & \textrm{if $sp > p\vee 1$.}
   \end{cases} 
 \end{equation}
Thus we have obtained \eqref{eq:prop-frac-class}, however with $\frac{1}{2} d^{(p\vee 1)-1}$
 instead of $d^{(p\vee 1)-2}$. For $d=1$ we are done, therefore from now on we assume that $d>1$.
 Observe that the last estimate in \eqref{eq:P-estimate}
 is inefficient, unless $q=1$. To improve it, we will consider $d$ paths, where we
 start moving first along one of the $d$ coordinates, instead of always the first.
 More specifically, let $\sigma$ denote the shift operator,
 \[
 \sigma(x_1,x_2,\ldots,x_d) = (x_2,\ldots, x_d, x_1),
 \]
 and let
 \[
 \pat^\beta(j,m) = \left( \sigma^{-\beta}(\pat_k(\sigma^\beta(j), \sigma^\beta(m)))
 \right)_{k=0,\ldots,|\pat(j,m)|}.
 \]
 In particular, $\pat^0(j,m)=\pat(j,m)$, while $\pat^\beta(j,m)$ is a~similarly defined path,
 where we first move along $(\beta+1)$th axis, and then subsequent ones (in a~circular way),
 postponing until the very end the movement along the first largest coordinate we encounter.
 For example, in $\Z_+^3$,
  \begin{align*}
  \pat^1((1,1,0), (0,0,2)) &= ( (1,1,0), (1,1,1), (1,1,2), (0,1,2), (0,0,2) );\\
  \pat^2((1,1,0), (0,0,2)) &= ( (1,1,0), (1,1,1), (1,1,2), (1,0,2), (0,0,2) ).
  \end{align*}
 Then we can write \eqref{eq:sumoverpath} for each family of the paths and sum over $\beta$, obtaining
  \begin{align}
    &d \sum_{j\in A_n} \sum_{m\in A_{n+k}} |u(j)-u(m)|^p \label{eq:sumwithd}\\
    &\leq
    ((2^{n+k} d)^{p-1} \vee 1) \sum_{j\in A_n} \sum_{m\in A_{n+k}}
    \sum_{\beta=0}^{d-1}
    \sum_{t=0}^{|\pat^\beta(j,m)|-1} \left|u\Big(\pat^\beta(j,m)_t\Big) - u\Big(\pat^\beta(j,m)_{t+1}\Big)\right|^p. \nonumber
  \end{align}
  Then, when rearraging the above sum like before, we will obtain the following number
  of pairs of points
  \begin{equation}\label{eq:pairsofpoints2}
    \sum_{\beta=0}^{d-1} 2^{(d+1)n}2^{kd} \cdot 6 \cdot 2^{1-(\beta+1)/2} \leq
    6\cdot 2^{(d+1)n}2^{kd} \frac{\sqrt{2}}{1-2^{-1/2}}
  \leq 6 \cdot  2^{(d+1)n}2^{kd} \cdot 5
  \end{equation}
  instead of \eqref{eq:pairsofpoints} and \eqref{eq:P-estimate}. We proceed as before; the difference is that we have
  an additional factor $d$ on left side of \eqref{eq:sumwithd}, and an additional factor $5$
  on right side of \eqref{eq:pairsofpoints2}. This yields the desired constant.
\end{proof}

\section{Lower bounds}
In this section, we provide some bounds for two sequences of functions, which will allow
us to prove optimality of exponents in main theorems.

\begin{proposition}\label{prop:un}
  Let $u_n:\Z_+^d\to \R$ be defined as
  \[
  u_n(x) =\begin{cases}
  1 & \textrm{if $\|x\|_\infty \leq n$;}\\
  0 & \textrm{otherwise.}
  \end{cases}
  \]
  Then for $t>0$ and $0<p<\infty$,
  \begin{align}
    \sum_{j\in \Z_+^d\setminus\{0\}} \frac{|u_n(j)|^p}{\|j\|_\infty^t} &\geq
     \begin{cases}
      \frac{d}{d-t} n^{d-t}, & \textrm{if $d-t\geq 1$,}\\
      \frac{d}{|d-t|} |(n+1)^{d-t}-1|, & \textrm{if $d-t<1$ and $d-t\neq 0$,}\\
      d \ln(n+1), &  \textrm{if $d-t = 0$;}
      \end{cases} \label{eq:u1LHS}\\
    \sum_{x\in  \Z_+^d} \; \sum_{\substack{y\in \Z_+^d\\ y\sim x}}|u(x)-u(y)|^p 
    &\leq
     \begin{cases}
       2dn^{d-1} + 2^{d+1} d n^{d-2} & \textrm{if $d\geq 2$,}\\
       2dn^{d-1} & \textrm{if $d=1$.} 
       \end{cases} \label{eq:u1RHS}
  \end{align}
\end{proposition}
\begin{proof}
  We have,
  \begin{align*}
    \sum_{j\in \Z_+^d\setminus\{0\}} \frac{|u_n(j)|^p}{\|j\|_\infty^t}
    &=
    \sum_{j=1}^n \frac{(j+1)^d - j^d}{j^t}
    \geq \sum_{j=1}^n \frac{dj^{d-1}}{j^t} = d \sum_{j=1}^n j^{d-1-t}.
  \end{align*}
  If $d-1-t\geq 0$, then we can estimate
  \[
  d \sum_{j=1}^n j^{d-1-t}  \geq d \int_0^n j^{d-1-t}\,dj =  \frac{d}{d-t} n^{d-t}.
  \]
  In the remaining case $d-1-t < 0$, we estimate
  \[
  d \sum_{j=1}^n j^{d-1-t}  \geq d \int_1^{n+1} j^{d-1-t}\,dj
  \]
  and obtain the desired estimate \eqref{eq:u1LHS}.

  Let
  \begin{equation}\label{eq:SkWk}
   \begin{split}
   S_k &:= \{ x\in \Z_+^d: \|x\|_\infty = k \},\\
   W_k &:= \{ x\in S_k: \textrm{exactly one of $x_j$'s equals $k$} \}.
   \end{split}
  \end{equation}
  We observe that
  \begin{equation}\label{eq:cardSkWk}
   \begin{split}
   \# W_k &= dk^{d-1},\\
   \# (S_k\setminus W_k) &= (k+1)^d - k^d - dk^{d-1}
   =
   \begin{cases}
     \sum_{i=0}^{d-2} \binom{d}{i} k^i \leq 2^d k^{d-2} & \textrm{if $d\geq 2$,}\\
     0 & \textrm{if $d=1$.}
     \end{cases}
   \end{split}
  \end{equation}
  
  To estimate the sum \eqref{eq:u1RHS}, we first observe that $|u(x)-u(y)|\neq 0$ only when $x\in S_n$ and $y \in S_{n+1}$ or vice versa.  For $x\in W_n$, it holds $\sum_{\substack{y\in \Z_+^d\\ y\sim x}}|u(x)-u(y)|^p =1$,
   while for $x\in S_n\setminus W_n$ it holds $\sum_{\substack{y\in \Z_+^d\\ y\sim x}}|u(x)-u(y)|^p \leq d$.
   Therefore,
   \begin{align*}
     \sum_{x\in  \Z_+^d} \; \sum_{\substack{y\in \Z_+^d\\ y\sim x}}|u(x)-u(y)|^p
     &\leq
     2 \sum_{x\in  S_n} \; \sum_{\substack{y\in S_{n+1} \\ y\sim x}}|u(x)-u(y)|^p \\
     &\leq 2\# W_n + 2d\cdot\, \#(S_n\setminus W_n)
   \end{align*}
   which together with \eqref{eq:cardSkWk} gives \eqref{eq:u1RHS}, as desired.
\end{proof}

\begin{proposition}\label{prop:vn}
  Let $v_n:\Z_+^d\to \R$ be defined as
  \[
  v_n(x) =\left( 1 - \frac{\|x\|_\infty}{n}\right)_+.
  \]
  Then for $0<p<\infty$,
  \begin{align}
    \sum_{j\in \Z_+^d\setminus\{0\}} \frac{|v_n(j)|^p}{\|j\|_\infty^t} &\geq
    dn^{d-t}
    \left( B(p+1,d-t) - \frac{1}{(d-t) \wedge 1} \frac{1}{n^{(d-t)\wedge 1}} \right), \quad \textrm{when $0<t<d$,} \label{eq:Hvn} \\
    \sum_{j\in \Z_+^d\setminus\{0\}} \frac{|1 - v_n(j)|^p}{\|j\|_\infty^t} &\geq
    \begin{cases}
      n^{d-t},&\quad \textrm{when $t>d$,}\\
      \infty,&\quad \textrm{when $t=d$,}
    \end{cases} \label{eq:H1vn} \\
    \sum_{x\in  \Z_+^d}  \sum_{\substack{y\in \Z_+^d\\ y\sim x}}|u(x)-u(y)|^p 
    &\leq
    2n^{d-p} + 
    \begin{cases}
       2^{d+1} n^{d-1-p} & \textrm{if $d\geq 2$,}\\
       0 & \textrm{if $d=1$.}
     \end{cases}
     \label{eq:Dvn}
  \end{align}
\end{proposition}

\begin{proof}
  We estimate
  \begin{align*}
    \sum_{j\in \Z_+^d\setminus\{0\}} \frac{|v_n(j)|^p}{\|j\|_\infty^t}
    &\geq
    \sum_{j=1}^n \frac{(1-\frac{j}{n})^p}{j^t} dj^{d-1}
    = dn^{d-t} \sum_{j=1}^{n-1} \frac{1}{n} \left(1-\frac{j}{n}\right)^p \left(\frac{j}{n}\right)^{d-1-t}
    =: dn^{d-t} S.
  \end{align*}
  We consider two cases. If $-1<d-1-t\leq 0$,
  then the function $(0,1]\ni x\mapsto (1-x)^px^{d-1-t}$
  is decreasing, hence
  \[
  S \geq \int_{\frac{1}{n}}^1 (1-x)^p x^{d-1-t}\,dx \geq B(p+1,d-t) - \frac{n^{t-d}}{d-t}.
  \]
  On the other hand, if $d-1-t> 0$, then the function $[0,1]\ni x\mapsto (1-x)^px^{d-1-t}$
  is first increasing, then decreasing. Assuming that the maximum of this function lies in some
  interval $[\frac{k}{n}, \frac{k+1}{n}]$, we may estimate
  \[
  S \geq \int_{[0,\frac{k}{n}]\cup [\frac{k+1}{n}]}  (1-x)^p x^{d-1-t}\,dx
  \geq B(p+1,d-t) - \frac{1}{n}.
  \]

  To prove \eqref{eq:H1vn}, we estimate,
  \begin{align*}
    \sum_{j\in \Z_+^d\setminus\{0\}} \frac{|1-v_n(j)|^p}{\|j\|_\infty^t}
    &\geq
    \sum_{j=n}^\infty \frac{(j+1)^d - j^d}{j^t}
    \geq
    \sum_{j=n}^\infty dj^{d-1-t} \\
    &\geq d \int_n^\infty j^{d-1-t}\,dj
    \geq
    \begin{cases}
      n^{d-t},&\quad \textrm{if $t>d$,}\\
      \infty,&\quad \textrm{if $t=d$.}
    \end{cases}
  \end{align*}
  
  To prove \eqref{eq:Dvn}, we use the notation \eqref{eq:SkWk} and observe that
\[
|u(x)-u(y)|^p =
\begin{cases}
  \frac{1}{n^p} & \textrm{if $x\in S_k$, $y\in S_{k+1}$ or vice versa, with $0\leq k < n$;}\\
  0 & \textrm{otherwise.}
\end{cases}
\]
  Therefore, using \eqref{eq:cardSkWk},
  \begin{align*}
    \sum_{x\in  \Z_+^d} \; \sum_{\substack{y\in \Z_+^d\\ y\sim x}}|u(x)-u(y)|^p
    &\leq
    2 \sum_{k=0}^{n-1} \; \sum_{x\in  S_k}  \sum_{\substack{y\in S_{k+1}\\ y\sim x}}|u(x)-u(y)|^p\\
    &=
    2 \sum_{k=0}^{n-1} \; \left(\sum_{x\in  W_k}  \sum_{\substack{y\in S_{k+1}\\ y\sim x}} \frac{1}{n^p}
    + \sum_{x\in  S_k\setminus W_k}  \sum_{\substack{y\in S_{k+1}\\ y\sim x}} \frac{1}{n^p} \right) \\
    &\leq
    \frac{2}{n^p} \sum_{k=0}^{n-1} \; dk^{d-1} + \frac{2}{n^p} \#(S_k\setminus W_k)\cdot d \\
    &\leq
    \frac{2d}{n^p} \sum_{k=1}^{n-1} k^{d-1} +
    \begin{cases}
      \frac{d2^{d+1}}{n^p} \sum_{k=1}^{n-1} k^{d-2} & \textrm{if $d\geq 2$,}\\
     0 & \textrm{if $d=1$.}
     \end{cases}
  \end{align*}
  From here, \eqref{eq:Dvn} follows by the following estimate, valid for $t\geq 0$,
  \[
  \sum_{k=1}^{n-1} k^t \leq \int_0^n k^t \,dk = \frac{n^{t+1}}{t+1}.\qedhere
  \]
\end{proof}

\section{Proofs of Theorems}

\begin{proof}[Proof of Theorem~\ref{thm:fracH}]
  (1)
  Let $sp<d$, $d-sp\geq \delta>0$ and let $K=K(s,p,\delta)\in \N$ satisfy
  \[
  2^{sp+1}(2^{p-1}\vee 1) 2^{-K\delta} \leq 1.
  \]
  Then $K$ satisfies also \eqref{eq:lemgamma}.
    Note that the constant in \eqref{eq:const-lemma} may be bounded by
    \[
    C(d,p,s,K)=\frac{2^{sp+1+Ksp} (2^{p-1}\vee 1)}{1-2^{-d}}
    \leq C(1,p,s,K) = 2^{sp+2+Ksp} (2^{p-1}\vee 1).
    \]
    Therefore, from Lemma~\ref{lem:main}
\begin{align*}
  \sum_{j\in \Z_+^d\setminus\{0\}} \frac{|u(j)|^p}{\|j\|_\infty^{sp}}
  &\leq C(1,p,s,K)
  \sum_{n=1}^\infty \sum_{j\in A_n} \sum_{m\in A_{n+K}} \frac{|u(j)-u(m)|^p}{\|j-m\|_\infty^{d+sp}} \\
  &\leq C(1,p,s,K)
    \sum_{j\in \Z_+^d \setminus\{0\}} \sum_{\substack{m\in \Z_+^d \setminus\{0\}\\ m\neq j}} \frac{|u(j)-u(m)|^p}{\|j-m\|_\infty^{sp+d}},
\end{align*}
for all functions $u:\Z_+^d\to \C$, for which the left side is finite.

(3)
  Let $sp>d$ and let $K\in \N$ satisfy \eqref{eq:lemgamma}. For any function $u:\Z_+^d \to \C$ such that $u(0)=0$,
  it holds, by Lemma~\ref{lem:main},
\begin{align*}
  \sum_{j\in \Z_+^d \setminus\{0\}} \frac{|u(j)|^p}{\|j\|_\infty^{sp}}
    &\leq
    C(d,p,s,K) \sum_{n=1}^\infty \sum_{j\in A_n} \sum_{m\in A_{n+K}}
    \frac{|u(j)-u(m)|^p}{ \|j-m\|_\infty^{sp+d}} \\
   &\quad + 2^{Kd+1} \sum_{\substack{j\in \Z_+^d\\ 1 \leq \|j\|_\infty < 2^{K}}} \frac{|u(j)-u(0)|^p}{\|j - 0\|_\infty^{sp+d}} \\
    &\leq
    C(d,p,s,K) \sum_{j\in \Z_+^d } \sum_{\substack{m\in \Z_+^d\\ m\neq j}}
    \frac{|u(j)-u(m)|^p}{ \|j-m\|_\infty^{sp+d}},
\end{align*}
because $2^{Kd+1} \leq C(d,p,s,K)$.

(2) Let $sp=d$ and $\varepsilon>0$. Put $s'=s+\varepsilon/p$, then $s'p>d$, therefore we may apply
\eqref{eq:thmfrac-sp-large} with $s'$ in place of $s$. This gives us the first inequality in
\eqref{eq:thmfrac-sp-d}, and the second one is trivial.

(4) First we prove \eqref{eq:thmfrac-sp-large} for $\Z^d$ in place of $\Z_+^d$.
For $\varepsilon \in \{-1,1\}^d$, let $\Z_\varepsilon^d = \bigtimes_{i=1}^d \varepsilon_i \Z_+$.
By a~variable substitution $\tilde{j} =(\varepsilon_1 j_1,\ldots,\varepsilon_d j_d)$ and $\tilde{m}$
instead of $j$ and~$m$ in \eqref{eq:thmfrac-sp-large} we obtain  \eqref{eq:thmfrac-sp-large}
for $\Z_\varepsilon^d$ in place of~$\Z_+^d$. Since $\bigcup_\varepsilon \Z_\varepsilon^d = \Z^d$, we obtain,
by summing,
\[
\sum_{j\in \Z^d \setminus\{0\}} \frac{|u(j)|^p}{\|j\|_\infty^{sp}}
    \leq c 2^d
    \sum_{j\in \Z^d } \sum_{\substack{m\in \Z^d\\ m\neq j}} \frac{|u(j)-u(m)|^p}{\|j-m\|_\infty^{sp+d}}
\]
for all functions $u:\Z^d \to \C$ such that $u(0)=0$; factor $2^d$ above comes from the fact
that the sets $\Z_\varepsilon^d$ overlap.

Again, \eqref{eq:thmfrac-sp-d} for $\Z^d$ follows from \eqref{eq:thmfrac-sp-large} for $\Z^d$
applied
to $s$ replaced by $s+\varepsilon/p$.

We could prove \eqref{eq:thmfrac-sp-small} for $\Z^d$ in place of $\Z_+^d$ in the same way
as \eqref{eq:thmfrac-sp-large},
but then the constant would depend on~$d$. Therefore we proceed in another way.
Let $Z_0 = \{0, -1,-2,\ldots\}$ and $Z_1 = \N = \{1,2,\ldots\}$.
For $\eta\in \{0,1\}^d$, we set $Z_\eta = \bigtimes_{i=1}^d Z_{\eta_i}$.
By substituting variables in \eqref{eq:thmfrac-sp-small}, we obtain
\[
\sum_{j\in \Z_\eta \setminus\{\eta\}} \frac{|u(j)|^p}{\|j - \eta \|_\infty^{sp}}
    \leq c
    \sum_{j\in \Z_\eta \setminus\{\eta\}} \sum_{\substack{m\in \Z_\eta \setminus\{\eta\}\\ m\neq j}} \frac{|u(j)-u(m)|^p}{\|j-m\|_\infty^{sp+d}},
    \]
    provided the left side is finite. Note that for $j\in \Z_\eta \setminus\{\eta\}$,
    \[
    \|j - \eta \|_\infty \leq \|j\|_\infty +  \|\eta \|_\infty^{sp} \leq  \|j\|_\infty + 1 \leq 2\|j\|_\infty.
    \]
    Since the sets $Z_\eta$ are pairwise disjoint and $\bigcup_\eta (\Z_\eta \setminus\{\eta\})
    = \Z^d \setminus \{0,1\}^d$, we obtain
    \begin{align*}
      \sum_{j\in \Z^d \setminus \{0,1\}^d} \frac{|u(j)|^p}{\|j \|_\infty^{sp}}
      &\leq
      2^{sp}
      \sum_\eta
\sum_{j\in \Z_\eta \setminus\{\eta\}} \frac{|u(j)|^p}{\|j - \eta \|_\infty^{sp}}
&\leq
c2^{sp}
    \sum_{j\in \Z^d \setminus\{0,1\}^d} \sum_{\substack{m\in \Z^ \setminus\{0,1\}^d\\ m\neq j}} \frac{|u(j)-u(m)|^p}{\|j-m\|_\infty^{sp+d}}.
    \end{align*}
    By adding \eqref{eq:thmfrac-sp-small} to both sides of the above inequality, we obtain
    \[
    \sum_{j\in \Z^d \setminus\{0\}} \frac{|u(j)|^p}{\|j\|_\infty^{sp}}
    \leq c(2^{sp}+1)
    \sum_{j\in \Z^d \setminus\{0\}} \sum_{\substack{m\in \Z^d \setminus\{0\}\\ m\neq j}} \frac{|u(j)-u(m)|^p}{\|j-m\|_\infty^{sp+d}}. \qedhere
    \]
\end{proof}

\begin{lemma}\label{lem:trivial}
Let $N\in \N$ and $B_N = \{x\in \Z_+^d : \|x\|_\infty \leq N\}$.
  There exists a~constant $c=c(p,d,N)$ such that
  \[
  \sum_{j\in B_N\setminus \{0\}} |u(j)|^p
  \leq c
  \sum_{j\in B_N}
  \sum_{\substack{k\in B_N \\ k\sim j}}
  |u(j)-u(k)|^p,
  \]
  whenever $u:\Z_+^d\to \C$ with $u(0)=0$.
\end{lemma}
\begin{proof}
  We prove by induction that for $r=1,2,\ldots,Nd$, there exists a~constant $c(p,d,N,r)$ such that
  \begin{equation}\label{eq:lemball}
  \sum_{\substack{j\in B_N\setminus \{0\} \\ \|j\|_1 \leq r}} |u(j)|^p
  \leq c(p,d,N,r)
  \sum_{\substack{j\in B_N \\ \|j\|_1 \leq r}}
  \sum_{\substack{k\in B_N \\ k\sim j \\ \|k\|_1 \leq r}}
  |u(j)-u(k)|^p,
  \end{equation}
  when $u(0)=0$.

  1. When $r=1$ and $\|j\|_1 = r = 1$, then $j\sim 0$, hence
  \[
  \sum_{\substack{j\in B_N\setminus \{0\} \\ \|j\|_1 \leq r}} |u(j)|^p
  = \sum_{\substack{j\in B_N \setminus \{0\} \\ \|j\|_1 \leq r}} |u(j) - u(0)|^p =
  \frac{1}{2}
  \sum_{\substack{j\in B_N \\ \|j\|_1 \leq r}}
  \sum_{\substack{k\in B_N \\ k\sim j \\ \|k\|_1 \leq r}}
  |u(j)-u(k)|^p,
  \]
  so $c(p,d,N,1) = 1/2$.

  2. Suppose that \eqref{eq:lemball} holds for some $r$.
  We have,
  \begin{equation}\label{eq:lemball2}
    \sum_{\substack{j\in B_N\setminus \{0\} \\ \|j\|_1 \leq r + 1}} |u(j)|^p
    =
    \sum_{\substack{j\in B_N\setminus \{0\} \\ \|j\|_1 \leq r}} |u(j)|^p
    +
    \sum_{\substack{j\in B_N \\ \|j\|_1 = r + 1}} |u(j)|^p.
    \end{equation}
  Note that if $j\in B_N$ and $\|j\|_1 = r + 1$, then there exists a~neighbour $k\sim j$
  with $k\in B_N$ and $\|k\|_1=r$.
  Therefore the second term in the above sum can be estimated as follows,
  \[
  \sum_{\substack{j\in B_N \\ \|j\|_1 = r + 1}} |u(j)|^p
  \leq
  (2^{p-1}\vee 1) \bigg(
   \sum_{\substack{j\in B_N \\ \|j\|_1 = r+1}}
  \sum_{\substack{k\in B_N \\ k\sim j \\ \|k\|_1 = r}}
  |u(j)-u(k)|^p + 
  \sum_{\substack{k\in B_N \\ \|k\|_1 = r}} |u(k)|^p \bigg)
  \]
  Coming back to \eqref{eq:lemball2},
  \begin{align*}
    \sum_{\substack{j\in B_N\setminus \{0\} \\ \|j\|_1 \leq r + 1}} |u(j)|^p
    &\leq
    (1+(2^{p-1}\vee 1)) \sum_{\substack{j\in B_N\setminus \{0\} \\ \|j\|_1 \leq r}} |u(j)|^p \\
    &\quad + (2^{p-1}\vee 1) 
   \sum_{\substack{j\in B_N \\ \|j\|_1 = r+1}}
  \sum_{\substack{k\in B_N \\ k\sim j \\ \|k\|_1 = r}}
  |u(j)-u(k)|^p \\
  &\leq
  c(p,d,N,r+1) 
   \sum_{\substack{j\in B_N \\ \|j\|_1 \leq r+1}}
  \sum_{\substack{k\in B_N \\ k\sim j \\ \|k\|_1 \leq r+1}}
  |u(j)-u(k)|^p,
  \end{align*}
  where
  \[
  c(p,d,N,r+1) =  (1+(2^{p-1}\vee 1))c(p,d,N,r) \vee 2^{p-1}\vee 1.
  \]
  Since $\|j\|_1 \leq Nd$ for $j\in B_N$, therefore the lemma follows from \eqref{eq:lemball}
  with $r=Nd$.
\end{proof}

\begin{proof}[Proof of Theorem~\ref{thm:classH}]
  First we prove (1) and (2).
  Suppose that $0<p\leq 1 < d$ or $1\leq p < d$ and $d-p\geq \delta>0$.
  In the former case, we put $\delta = 1$ and then $d-1 \geq 1 = \delta$.
  Let $K=K(s,p,\delta)\in \N$ satisfy
  \[
  2^{sp+1}(2^{p-1}\vee 1) 2^{-K\delta} \leq 1.
  \]
  Take $s=\frac{1}{p} \vee 1$, then $sp = p\vee 1 < d$ and $d-sp \geq \delta$, therefore
  $K$ satisfies \eqref{eq:lemgamma}. Hence from inequality \eqref{eq:lemH-small-sp}
  in Lemma~\ref{lem:main},
  \begin{equation}\label{eq:proofappllemma}
  \sum_{j\in \Z_+^d\setminus\{0\}} \frac{|u(j)|^p}{\|j\|_\infty^{sp}} \leq C(d,p,s,K)
  \sum_{n=1}^\infty \sum_{j\in A_n} \sum_{m\in A_{n+K}} \frac{|u(j)-u(m)|^p}{2^{(n+K)(d+sp)}}
  \end{equation}
for all functions $u:\Z_+^d\to\C$ such that the left side above is finite.
However,
\begin{equation}\label{eq:proofconstestimate}
C(d,p,s,K) =\frac{2^{sp+1+Ksp} (2^{p-1}\vee 1)}{1-2^{-d}} \leq 2C(1,p,s,K) =: \tilde{C}(p,s,\delta).
\end{equation}
Right side of \eqref{eq:proofappllemma} may be estimated using Proposition~\ref{prop:frac-class},
\[
\sum_{n=1}^\infty \sum_{j\in A_n} \sum_{m\in A_{n+K}} \frac{|u(j)-u(m)|^p}{2^{(n+K)(d+sp)}}
  \leq 
  C_L(K,s,p)\;   d^{(p\vee 1)-2} \!\!
  \sum_{x \in \Z_+^d\setminus\{0\}}
  \sum_{\substack{y \in \Z_+^d\setminus\{0\} \\ y\sim x}}  \frac{|u(x)-u(y)|^p}{\|x\|_\infty^{sp-(p \vee 1)}}\;.
\]
But $sp-(p \vee 1) = 0$, so \eqref{eq:thmlapl-p-small} and \eqref{eq:thmlapl-p-medium}
follow from \eqref{eq:proofappllemma}, \eqref{eq:proofconstestimate} and above estimate.

We proceed to prove (3) and (4).
Suppose that $d<p$ or $d=p$ and $\varepsilon>0$.
Take
\[
s = \begin{cases}
  1,&\textrm{if $d<p$,}\\
  1 + \dfrac{\varepsilon}{p}, &\textrm{if $d=p$.}
\end{cases}
\]
Then $sp>d$. Let $K=K(s,p,d)\in \N$ satisfy \eqref{eq:lemgamma}.
From inequality  \eqref{eq:lemH-large-sp} in Lemma~\ref{lem:main},
\begin{equation}\label{eq:proofspd}
\sum_{j\in \Z_+^d\setminus \{0\}} \frac{|u(j)|^p}{\|j\|_\infty^{sp}}
    \leq
    C(d,p,s,K) \sum_{n=1}^\infty \sum_{j\in A_n} \sum_{m\in A_{n+K}}
    \frac{|u(j)-u(m)|^p}{2^{(n+K)(d+sp)}}
    + 2\!\!\!\!\!\!\!\!\sum_{\substack{j\in \Z_+^d\\ 1 \leq \|j\|_\infty < 2^{K}}} \frac{|u(j)|^p}{\|j\|_\infty^{sp}}
\end{equation}
for all functions $u:\Z_+^d\to\C$. The first term above may be estimated using Proposition~\ref{prop:frac-class},
\[
  \sum_{n=1}^\infty \sum_{j\in A_n} \sum_{m\in A_{n+k}} \frac{|u(j)-u(m)|^p}{2^{(n+K)(d+sp)}}
  \leq 
  C_L(K,s,p)\;   d^{p-2} \!\!
  \sum_{x \in \Z_+^d\setminus\{0\}}
  \sum_{\substack{y \in \Z_+^d\setminus\{0\} \\ y\sim x}}  \frac{|u(x)-u(y)|^p}{\|x\|_\infty^{sp-p}}\;.
\]
In case $d<p$, it holds $sp-p = 0$, while in the case $d=p$, $sp-p = \varepsilon$.

The last term in \eqref{eq:proofspd} may be estimated using Lemma~\ref{lem:trivial},
\begin{align*}
\sum_{\substack{j\in \Z_+^d\\ 1 \leq \|j\|_\infty < 2^{K}}} \frac{|u(j)|^p}{\|j\|_\infty^{sp}}
&\leq c(p,d,K)
  \sum_{\substack{ j\in \Z_+^d \\ \|j\|_\infty < 2^{K}}}
  \sum_{\substack{k\in\Z_+^d \\ \|j\|_\infty < 2^{K} \\ k\sim j}}
  |u(j)-u(k)|^p \\
&\leq c'(p,d,K,\varepsilon)
  \sum_{\substack{ j\in \Z_+^d \\ \|j\|_\infty < 2^{K}}}
  \sum_{\substack{k\in\Z_+^d \\ \|k\|_\infty < 2^{K} \\ k\sim j}}
  \frac{|u(j)-u(k)|^p}{(\|j\|_\infty \vee \|k\|_\infty)^\varepsilon}.
\end{align*}

To prove (\ref{item:deq1_ple1}), we apply \eqref{eq:thmlapl-p-large} with $1+\varepsilon$ in place
of $p$ and $|u|^{\frac{p}{1+\varepsilon}}$ in place of $u$ and obtain,
\[
 \sum_{j\in \Z_+ \setminus\{0\}} \frac{|u(j)|^p}{\|j\|_\infty^{1+\varepsilon}}
    \leq c
    \sum_{j\in \Z_+\setminus\{0\} } \sum_{\substack{k\in \Z_+ \\ k\sim j}}
    ||u(j)|^{\frac{p}{1+\varepsilon}}-|u(k)|^{\frac{p}{1+\varepsilon}}|^{1+\varepsilon},
\]
for all functions $u:\Z_+ \to \C$ with $u(0)=0$. From there, the result follows, since
\[
\left||u(j)|^{\frac{p}{1+\varepsilon}}-|u(k)|^{\frac{p}{1+\varepsilon}}\right|^{1+\varepsilon}
\leq \Big||u(j)|-|u(k)|\Big|^p \leq |u(j)-u(k)|^p
\]
by concavity of $[0,\infty) \ni x\mapsto x^{\frac{p}{1+\varepsilon}}$ and triangle inequality.
  \footnote{Note that we could similarly prove \eqref{eq:thmlapl-p-small} for $p<1$, using
  \eqref{eq:thmlapl-p-small} for $p=1$.}

To prove (\ref{item:optimal}), let $t$, $p>0$. Suppose that for some $C<\infty$ the inequality
\begin{equation}\label{eq:thmlapl-inverse}
  \sum_{j\in \Z_+^d \setminus\{0\}} \frac{|u(j)|^p}{\|j\|_\infty^t}
  \leq C
  \sum_{j\in \Z_+^d \setminus\{0\}} \sum_{\substack{k\in \Z_+^d\setminus\{0\} \\ k\sim j}} |u(j)-u(k)|^p
\end{equation}
holds 
for all functions $u:\Z_+^d \to \C$ for which the left side of \eqref{eq:thmlapl-inverse} is finite,
or which satisfy $u(0)=0$, if $p<d$ or $d\leq p$, respectively.

If $p\leq 1 < d$ and $t\leq 1$ (so $d-t\geq 1$), then taking in \eqref{eq:thmlapl-inverse}
$u=u_n$ from Proposition~\ref{prop:un} we see
that
\[
\frac{d}{d-t} n^{d-t} \leq C(dn^{d-1} + 2^d d n^{d-2}),
\]
for every $n\in \N$, so it must hold $t=1$. Therefore, exponent $1$ at $\|j\|_\infty$
on the left side of \eqref{eq:thmlapl-p-small} is indeed the smallest possible.

Similarly, when $p\leq 1 = d$ and $t=1$, then  the same choice of $u=u_n$ as before gives
$\ln(n+1) \leq 2C$, which is a~contradiction. Hence \eqref{eq:thmlapl-p-le-d-eq-1}
does not hold with $\varepsilon=0$.

If $1<p<d$ and $t<d$, then  taking in \eqref{eq:thmlapl-inverse}
$u=v_n$ from Proposition~\ref{prop:vn} we see that
\[
    c(p,d,t) n^{d-t} \leq  C(n^{d-p} +  2^d n^{d-1-p})
\]
with some $c(p,d,t)>0$ and every $n$.
Therefore, $t\geq p$ and consequently, exponent $p$ at $\|j\|_\infty$
on the left side of \eqref{eq:thmlapl-p-medium} is  the smallest possible.

We proceed similarly in \eqref{eq:thmlapl-p-large} and \eqref{eq:thmlapl-p-d}, taking
in  \eqref{eq:thmlapl-inverse}
$u=1-v_n$ with $v_n$ as in~Proposition~\ref{prop:vn} to prove optimality of the exponent
$p$ at $\|j\|_\infty$ on the left side of \eqref{eq:thmlapl-p-large}, as well as the fact
that \eqref{eq:thmlapl-p-d} fails to hold for $\varepsilon=0$.
Finally, to prove optimality of the exponent $-1$ at $d^{-1}$ in \eqref{eq:thmlapl-p-small},
suppose  that  \eqref{eq:thmlapl-p-small} holds with $d^{-1-\varepsilon}$ on right side.
For a~function $u$ defined by $u((1,0,\ldots,0))=1$ and $u=0$ otherwise,
the left side of \eqref{eq:thmlapl-p-small} equals $1$, while the right side equals $cd^{-1-\varepsilon}\cdot d=cd^{-\varepsilon}$. Thefefore $\varepsilon\geq 0$.

Finally,  (\ref{item:Z}) may be proved in a~similar manner as~(4) in Theorem~\ref{thm:fracH}, we omit the details.
\end{proof}

\subsection*{Acknowledgments.}
I would like to thank Luz Roncal, Florian Fischer and Matthias Keller for helpful discussions.
I am grateful to the anonymous reviewer for careful reading of the manuscript and spotting
an error in the preliminary version of the proof of Proposition~\ref{prop:frac-class}.

  \enlargethispage*{4 \baselineskip}

\def\cprime{$'$} \def\cprime{$'$}
  \def\polhk#1{\setbox0=\hbox{#1}{\ooalign{\hidewidth
  \lower1.5ex\hbox{`}\hidewidth\crcr\unhbox0}}}
  \def\polhk#1{\setbox0=\hbox{#1}{\ooalign{\hidewidth
  \lower1.5ex\hbox{`}\hidewidth\crcr\unhbox0}}}
  \def\polhk#1{\setbox0=\hbox{#1}{\ooalign{\hidewidth
  \lower1.5ex\hbox{`}\hidewidth\crcr\unhbox0}}} \def\cprime{$'$}
  \def\cprime{$'$} \def\cprime{$'$} \def\cprime{$'$} \def\cprime{$'$}
  \def\cprime{$'$}

\end{document}